\documentclass[leqno]{amsart}
\usepackage{times}
\usepackage{amsfonts,amssymb,amsmath,amsgen,amsthm,latexsym}
\usepackage{epsf,exscale,times}
\usepackage[dvips]{graphicx}
\usepackage{graphics,epsfig}
\usepackage{epsfig,rotate,color}
\usepackage{subfigure}  
\usepackage{hyperref}

\theoremstyle{plain}
\newtheorem{theorem}{Theorem}[section]

\newtheorem{lemma}[theorem]{Lemma}
\newtheorem{corollary}[theorem]{Corollary}
\newtheorem{proposition}[theorem]{Proposition}

\theoremstyle{remark}
\newtheorem{remark}[theorem]{Remark}

\def\R{{\mathbf R}}% real numbers
\def\N{{\mathbf N}}% nonnegative integers
\def\O{\mathcal O}
\newcommand{\SR}{\mathcal{S}(\mathbf{R})}

\def\I{\mathcal{I}}

\def\F{\mathcal{F}}

\def\({\left(}
\def\){\right)}
\def\<{\left\langle}
\def\>{\right\rangle}
\def\le{\leqslant}
\def\ge{\geqslant}

\def\d{{\partial}}
\def\eps{\varepsilon}

\DeclareMathOperator{\RE}{Re}

\numberwithin{equation}{section}

\begin{document}

\title[Splitting  for  the Fowler equation]{Splitting methods for
 the nonlocal Fowler equation} 

\author[A. Bouharguane]{Afaf Bouharguane}
\address{Laboratoire Jean Kuntzmann, Univ. Joseph Fourier \\
38041 Grenoble Cedex 9, France.}
\email{afaf.bouharguane@imag.fr}

\author[R. Carles]{R\'emi Carles}
\address{CNRS \& Univ. Montpellier~2 \\Math\'ematiques
\\CC~051\\34095 Montpellier\\France}
\email{Remi.Carles@math.cnrs.fr}
 
\begin{abstract}
We consider a nonlocal scalar conservation law proposed by Andrew
C. Fowler to describe the dynamics of dunes, and we develop a
numerical procedure based on splitting methods to approximate its
solutions.  We begin by proving the convergence of the well-known Lie
formula, which is an approximation of the exact solution of order one in
time. We next use  
the split-step Fourier method to approximate the continuous problem
using the fast Fourier transform and the finite difference method. % We show that the well-known Lie formula is an approximation of the exact solution of order 1 in time.  
%and Strang formulae are approximations of the exact solution of order
%1 and 2 in time.  
%Some examples are provided in which the numerical solutions are
%obtained and compared against exact solutions.  
Our numerical experiments confirm the theoretical results. 
\end{abstract}
\thanks{This work was supported by the French ANR project
  MATHOCEAN, ANR-08-BLAN-0301-02.} 
\maketitle

\section{Introduction}

We consider the Fowler equation \cite{Fo-p,Fo01}:
\begin{equation}
\left\{
    \begin{aligned}
        &\partial_t u(t,x) + \partial_x\left(\frac{u^2}{2}\right) (t,x)
        + \I [u(t,\cdot)] (x)  
- \partial_{x}^2 u(t,x) = 0,\quad  x \in \R, t>0, \\
        &u(0,x) = u_0(x),\quad  x \in \R,
    \end{aligned}
\right.
\label{fowlereqn5}
\end{equation} 
where $u=u(t,x)$ represents the dune height and $\I$
is a nonlocal operator defined as follows: for any Schwartz
function $\varphi \in \SR$ and any $x \in \R$,
\begin{equation}
\I [\varphi] (x) := \int_{0}^{+\infty} |\xi|^{-\frac{1}{3}}
\varphi''(x-\xi) \, d\xi . \label{nonlocalterm4}
\end{equation}
We refer to \cite{AAI10,AA09,Bo12} for theoretical results on this
equation.

\begin{remark}
The nonlocal term $\I$ is anti-diffusive. Indeed, it has been proved
in \cite{AAI10} that   
\begin{equation}\label{fourier}
\F\left( \I[\varphi]\right)(\xi) = - 4 \pi^2
\Gamma\left(\frac{2}{3}\right) \left(\frac{1}{2}-i \, \mbox{sgn}(\xi)
  \frac{\sqrt{3}}{2}\right)|\xi|^{4/3},
%, \nonumber \\
%&= - a_\I |\xi|^{4/3} + i b_\I \xi |\xi|^{1/3},
\end{equation}
%with $a_\I = 2 \pi^2 \Gamma \(\frac{2}{3}\)$, $b_\I
%=2 \pi^2 \sqrt{3} \Gamma \(\frac{2}{3}\)$ and 
where $\F$ denotes
the Fourier transform normalized in \eqref{TF}. Thus, 
$\I$ can been seen as a fractional power of order $2/3$ of the
Laplacian, with the ``bad'' sign. It will be clear from the
analysis below that our results can easily be extended to the case
where $\I$ is replaced with a Fourier multiplier homogeneous of degree
$\lambda\in ]0,2[$, as in \cite{ABC12}, and not only $\lambda=4/3$. 
\end{remark}

We assume that the initial data $u_0$ belongs to $H^3(\R)$, and thus
\eqref{fowlereqn5} has a unique solution belonging to $C([0,t],
H^3(\R))$ for all $t>0$, from \cite{AAI10}.  
We will denote $u(t, \cdot)$ by $S^t u_0$; $S^t$ maps $ H^3(\R)$ to
itself. Duhamel's formula
for the continuous problem \eqref{fowlereqn5} reads
\begin{equation}
u(t, \cdot) := S^t u_0= K(t,\cdot) \ast u_0 - \frac{1}{2} \int_0^t \partial_x K(t-s, \cdot) \ast \left( S^s u_0\right)^2 \, ds, 
\label{duhamel5}
\end{equation}
where $K(t, \cdot) = \F^{-1}\left( e^{-t \psi_\I}\right) $ is the kernel of the operator $\I-\partial_{x}^2$, and $\psi_\I$ is defined by
\begin{equation}
\psi_\I(\xi) = 4 \pi^2 \xi^2- a_\I |\xi|^{4/3} + i b_\I \xi |\xi|^{1/3},
\end{equation}
where $a_\I,b_\I$ are positive constants. 
\smallbreak

Recently, to solve the Fowler equation some numerical experiments have
been performed using mainly finite difference approximation schemes
\cite{AB-p,ABC12}. However, these schemes are not effective because if
we opt for an explicit scheme, numerical stability requires that the
time step $\Delta t$ is limited by $O(\Delta x^2).$  And, if we choose
an implicit scheme, we have to solve a large system which is a
computationally expensive operation. Thus, the splitting method
becomes an interesting alternative to solve the Fowler model.  
To our knowledge, there is no convergence result in the literature for
the splitting method associated to the Fowler equation.  
This method is more commonly used to split different physical terms,
such as reaction and diffusion terms, see for instance
\cite{RS09}. Splitting methods have also been employed for solving a
wide range of nonlinear wave equations. The basic idea of this method
is to decompose the original problem into sub-problems and then to
approximate the solution of the original problem by solving
successively the
sub-problems. Various versions of this method have
been developed for the nonlinear Schr\"{o}dinger, Korteweg-de-Vries
and modified Korteweg-de-Vries equations, see for instance
\cite{HKLR10,Lu08,Sa07,TA84}. \\ 
For the Fowler model \eqref{fowlereqn5}, we consider, separately, the
linear Cauchy problem 
\begin{equation}
\frac{\partial v}{\partial t} +  \I[v(t, \cdot)] - \eta
\, \partial_{x}^2 v = 0; \quad v(0,x) = v_0(x), 
\label{nlocal}
\end{equation} 
and 
the nonlinear Cauchy problem 
\begin{equation}
\frac{\partial w}{\partial t} + \partial_x \left( \frac{w^2}{2} \right) - \varepsilon \, \partial_{x}^2 w = 0;\quad
        w(0,x) = w_0(x),
\label{burgers}
\end{equation} 
where $ \varepsilon, \eta$ are fixed \emph{positive} parameters such that
$\varepsilon + \eta = 1$. 
Equation \eqref{burgers} is simply the viscous Burgers' equation. We
denote by $X^t$ and $Y^t$, respectively, the evolution operator
associated with \eqref{nlocal} and \eqref{burgers}: 
\begin{equation*}
v(t, \cdot) := X^t v_0 = D(t, \cdot)\ast v_0,
\end{equation*}
where $ D(t, \cdot) = \F^{-1}\left( e^{-t \, \phi_\I}\right) $ with
$\phi_\I(\xi) = 4\pi^2 \eta \xi^2 - a_\I |\xi|^{4/3} + b_\I \xi
|\xi|^{1/3}$, and
\begin{equation}
w(t, \cdot) := Y^t w_0 = G(t, \cdot) \ast w_0 - \frac{1}{2} \int_0^{t} \partial_x G(t-s, \cdot) \ast \left(Y^s w_0 \right)^2 \, ds, 
\label{burgerduhamel}
\end{equation}
where $G$ is the heat kernel defined by
\begin{equation*}
G(t, \cdot) = \F^{-1}(e^{-t(4 \pi^2 \varepsilon |.|^2) }) = \frac{1}{\sqrt{4\pi \varepsilon t}} e^{-\frac{|.|^2}{4\varepsilon t}}.
\end{equation*}
Furthermore, the following $L^2$-estimate holds
\begin{equation}
\|Y^t w\|_{L^{2}(\R)} \le \|w\|_{L^{2}(\R)}.
\label{estimapriori1}
\end{equation}
Let us explain the choice of this decomposition. First, we can remark
that if we do not consider the nonlinear term in \eqref{fowlereqn5},
the analytical solutions are available using the Fourier
transform. Thus, the linear part may be computed efficiently using a
fast Fourier transform (hereafter FFT) algorithms. Note also that the
Laplacian and the fractional term $\I$ cannot be treated
separately. Indeed, the equation $u_t + \I[u] = 0$ is ill-posed. We
next decide to handle the nonlinear term by adding a bit of
viscosity in order to avoid shock problems in the standard Burgers'
equation. Therefore, the splitting approach presented in this article
differs from e.g. the one analyzed in \cite{Fa09}, which corresponds to
assuming $\eps=0$ in the above definitions. The splitting operators
associated to this approach when $\I=0$ (which amount to considering
alternatively the heat equation and the Burgers equation) have been
studied in \cite{HLR-p} (as well as other equations involving the
Burgers nonlinearity, such as the KdV equation, see also
\cite{HKRT11}). The main difference with the result
presented in \cite{HLR-p} is that the operator $\I$ is not a
differential operator, so its action on nonlinear terms is rather
involved
,
while this point would be needed to compute the Lie commutator between
the two operators
\begin{equation*}
  A(v) = \eta \d_x^2 v - \I[v],\quad B(v) = -v\d_x v +\eps^2\d_x^2 v,
\end{equation*}
as in e.g. \cite{HLR-p,Lu08}. Note that since we consider the Lie
splitting operator, we could in principle use 
 the exact formula established in
\cite{DeTh-p} for the local error. But then again, we face the problem
to compute $[A,B]$, which is the only term that we cannot estimate in
\cite[Theorem~1]{DeTh-p}. We finally point out that we use smoothing
effects associated to the 
viscous Burgers equation; see Corollaries~\ref{cor:condest} and
\ref{cor:split1}.  
%It should be clear that using the
%methods from \cite{HKRT11,HLR-p}, the analysis of splitting operators
%in the limiting case $\eps=0$ (Burgers equation instead of the viscous
%Burgers equation) could be achieved. 

We  motivate this choice by the presence of artificial diffusion in classical numerical schemes used to solve the convection equations.
An alternative to reduce this effect is to consider numerical schemes
of high order which are usually  computationally expensive and do not
seem to be very useful for the Fowler model because of the diffusion
term.  
\smallbreak

We consider the Lie formula defined by
\begin{equation}\label{eq:Lie}
Z^t_L = X^t Y^t. 
\end{equation}
%This approximation is of order one, at least formally. 
The alternative definition $Z^t_L =Y^t X^t$ could be studied as well,
leading to a similar result. 
Also, the
following evolution operators 
\begin{equation*}
Z^t_S = X^{t/2 } Y^t X^{t/2 } \quad \text{or}  \quad Z^t_S = Y^{t/2} X^t Y^{t/2},
\end{equation*}
 corresponding to  the Strang method \cite{St68} could be considered. Following
 the computations detailed in the present paper for the case
 \eqref{eq:Lie}, it would be possible 
 to show that the other Lie formula generates a scheme of order one,
 and to prove that 
 the Strang method is of order two (for smooth initial data), in the
 same fashion as in, e.g., 
 \cite{BBD02,Lu08}. This fact is simply illustrated 
 numerically in Section~\ref{numerique}, to avoid a lengthy
 presentation. With $Z_L^t$ given by \eqref{eq:Lie}, our main result is:
\begin{theorem}\label{theoreme1}
For all $u_0 \in H^3(\R)$  and for all $T>0$, there exist positive
constants $c_1,c_2$ and $\Delta t_0$ such that for all $ \Delta t \in
]0,\Delta t_0]$ and for all $n \in \N$ such that $ 0 \le n
\Delta t \le T$, 
\begin{equation*}
\| (Z_L^{\Delta t})^{n } u_0 - S^{n \Delta t} u_0 \|_{L^{2}(\R)} \le
c_1 \,  \Delta t \quad \text{and}\quad \|(Z_L^{\Delta t})^{n }
u_0\|_{H^3(\R)}\le c_2.
\end{equation*}
Here,
$c_1,c_2$  and $\Delta t_0$ depend only
on
$T$, $\rho=\max_{t\in [0,T]} \|S^t u_0\|_{H^2(\R)}$, and $\|u_0\|_{H^3(\R)}$. 
\end{theorem}
\begin{remark}
  It will follow from Lemma~\ref{lemme8} that
  \begin{equation*} 
  \rho=\max_{t\in [0,T]} \|S^t u_0\|_{H^2(\R)}\le
  C_T(\|u_0\|_{H^1(\R)})\|u_0\|_{H^2(\R)},
  \end{equation*}
for some nonlinear (increasing) function $C_T$ depending on $T$. 
\end{remark}
\smallbreak

In this paper, we begin by estimating the $L^2$-stability for error
propagation.  
We next prove that the local error of the Lie formula is an
approximation of order two in time.  
%To show this result we use tools from operator theory, like in the
%paper of Besse et al. \cite{BBD02}. 
Finally we prove that this evolution operator represents a good
approximation, of order one in time,  of the evolution operator $S^t$
in the  sense of Theorem~\ref{theoreme1}. \\
%Finally we prove that this evolution operator represents a good
%approximation, of order one in time,  of the evolution operator $S^t$
%in the following sense: 
%\begin{equation*} 
%S^t \approx \left[ Z_L^{t/N} \right]^{N}, 
%\end{equation*} 
%for $N$ large and $t$ fixed. For that, we use the standard argument of
%Lady Windermere's fan.  
Furthermore, we apply Lie and Strang approximations in order to make
some numerical simulations using the split-step Fourier experiments.\\ 
%show numerically that the approximation of the evolution operator $S^t$ is in order $\Delta t$ for the nonlinear equation \eqref{fowlereqn5} using the split-step Fourier experiments.\\

This paper is organised as follows. In the next section we give some
properties related to the kernels $G$ and $K$, and we prove two fractional Gronwall Lemmas. In Section~\ref{sectionlemme},  
we establish some estimates on $X^t$,
$Y^t$, $Z_L^t$ and $S^t$. In Section~\ref{errorlocal}, we prove a
local $L^2$ error estimate.  Theorem~\ref{theoreme1} is proved in
Section~\ref{sec:proof}.   
 We finally perform some numerical experiments which show that the Lie
 and Strang methods have a convergence rate in $\O\left(
   \Delta t\right) $ and  $ \O\left(\Delta t^2
 \right))$, respectively. 

\bigskip 

 \noindent \textbf{Notations.}\\
%- The norm of a measurable function $f \in L^{p}(\R)$ is written $\|f
%\|^p_{L^{p}(\R)} = \int_{\R} |f(x)|^{p} \, dx $ for $1 \le p <
%\infty$, and 
%$\|f\|_{L^{\infty}(\R)} = \mbox{ess} \sup_\R |f|.$ \\ 
- We denote by $\F$ the Fourier transform of $f$ which is defined by: for all $\xi \in \R$,
\begin{equation}
\label{TF}
\F f (\xi)=\hat f(\xi):= \int_{\R} e^{-2i\pi x \xi} f(x) dx.
\end{equation}
We denote by $\F^{-1}$ its inverse. \\
%- The Schwartz space of rapidly decreasing functions on $\R$ is
%denoted by $\SR .$ \\ 
%- We write $C^{k}(\R) = \lbrace f: \R \rightarrow \mathbb{C}; f, f', \cdots, f^{(k) } \mbox{ are continous on } \R \rbrace $.\\
- We denote by $C_T(c_1,c_2, \cdots)$ a generic constant, strictly
positive, which depends on parameters $c_1,c_2, \cdots$, and $T$. $C$
is assumed to be a monotone increasing function of its arguments. \
%- We denote $H^{s}(\R)$ ($s \in \R$) the usual Sobolev Space
%$H^{s}(\R) = \left\lbrace u \in \mathcal{S}'(\R), \| u\|_{H^s} <
%\infty \right\rbrace $, where $\| u\|_{H^s}^2 = \int_\R \left(1+
%|\xi|^2 \right)^s |\hat{u}(\xi)|^2 \, d\xi  $. 

\section{Preliminaries} 

\subsection{Properties of the kernels}
\label{sec:kernels}

We begin by recalling the properties of kernels involved in the
present analysis.

\begin{proposition}[Main properties of $K$, \cite{AAI10}] \label{kernel2}
The kernel $K$ satisfies:
\begin{enumerate}
\item $\forall t>0$, $K(t,\cdot) \in L^{1}\left(\R\right)$ and $K \in
  C^{\infty}\left(]0,\infty[\times \R\right)$. 
\item $\forall s,t >0, \, \, K(s,\cdot)\ast K(t,\cdot)=K(s+t,\cdot)$.
	     % $\forall u_{0} \in L^2\left(\R\right), \, \, \lim_{t
             % \to 0} K\left(t,\cdot \right)*u_{0}=u_{0} \hspace{0.2
             % cm} in \hspace{0.2 cm} L^2\left(\R\right)$, 
\item $ \forall T>0, \exists C_{T} > 0$ such that for all 
$ t \in ]0,T]$,  $\| \partial_{x}K\left(t, \cdot
\right)\|_{L^2\left(\R\right)}\le C_{T} \, t^{-3/4}$.  
\item $ \forall T>0, \exists C_{T} > 0$ such that for all  
$t \in ]0,T]$, $\| \partial_{x}K\left(t,\cdot
\right)\|_{L^1\left(\R\right)}\le C_{T} \, t^{-1/2}$. 
%\item $\forall s,t >0 $, $K(s,\cdot) \ast \partial_x K(t,\cdot) = \partial_{x} K(t-s, \cdot) $.
\item For any $u_0 \in L^2(\R)$ and $t > 0$, 
\begin{equation*}
\|K(t,\cdot) \ast u_0\|_{L^2(\R)} \le e^{\alpha_0 t} \|u_0\|_{L^2(\R)},
\end{equation*}
where $\alpha_0 = -\min \RE (\psi_\I) >0.$
\end{enumerate}
\end{proposition}

%It is well-known that the heat kernel satisfies the following properties. We refer the reader to the literature for more details about proofs. 

\begin{proposition}[Main properties of $G$, \cite{DI04}] 
\label{heat}
The kernel $G$ satisfies:
\begin{enumerate}
\item $G \in C^{\infty}\left(]0,\infty[\times \R\right)$.
\item $\forall s,t >0$, $ G(s,\cdot)\ast G(t,\cdot)=G(s+t,\cdot)$.
	     % $\forall u_{0} \in L^2\left(\R\right), \, \, \lim_{t
             % \to 0} \left(t,\cdot \right)*u_{0}=u_{0} \hspace{0.2
             % cm} in \hspace{0.2 cm} L^2\left(\R\right)$,
\item $\forall t >0$, $\| G\left(t,\cdot
  \right)\|_{L^1\left(\R\right)} =1$.
    \item $\exists C_{0}>0$ such that for all
$ t >0$, $\|\partial_{x}G\left(t, \cdot
\right)\|_{L^2\left(\R\right)}\le C_{0} t^{-3/4}$.
\item  $\exists C_{1}>0$ such that for all
$t >0$, $\| \partial_{x}G\left(t, \cdot
\right)\|_{L^1\left(\R\right)}\le C_{1} t^{-1/2}$. 
\end{enumerate}
\end{proposition}

\begin{remark}\label{remarkconvolution}
The kernel $D$ of $ \, \I-\eta \, \partial_{x}^2$ has similar
properties to the kernel $K$. Moreover, for all $t>0$, we have  
\begin{equation}
D(t, \cdot) \ast G(t,\cdot) = K(t, \cdot).
\end{equation}
 %We can remark that the main difference between the kernel $K$ (or $D$) and the heat kernel is that the to estimate in  
\end{remark}

\subsection{Fractional Gronwall  lemmas}
\label{sec:gronwall}

\begin{lemma}[Fractional Gronwall Lemma]
\label{lemme1}
Let $\phi:[0,T] \to \R_+$ be a bounded measurable function,
and suppose that there are positive constants $A,L$ and $\theta \in
]0,1[$ such that for all $t \in [0,T]$,
\begin{equation}
\phi(t) \le A + L \frac{d^{-\theta}}{d t^{-\theta}}\phi(t),
\label{ingegality1}
\end{equation}
where $\bigskip \frac{d^{-\theta}}{d t^{-\theta}}$ is the
Riemann--Liouville operator defined by 
\begin{equation*}
\frac{d^{-\theta}}{d t^{-\theta}}\phi(t) = \frac{1}{\Gamma(\theta)}
\int_0^t (t-s)^{\theta-1} \phi(s) \, ds. 
\end{equation*}
Then there exists $C_T(\theta)$ such that 
%\begin{equation}
%\phi(t) \le C_T A + L^4 C \int_0^t \phi(s) \, ds,
%\end{equation}
\begin{equation*}
\phi(t) \le  e^{C_T(\theta)t}A,\quad \forall t\in [0,T].
\end{equation*}
\end{lemma}

\begin{proof}
The proof of this Lemma is well-known and is based on
  an iteration argument; see for instance
  \cite[Lemma~4.3]{Diethelm2004}  or \cite[Corollary 2]{Ye2007}. We
  sketch the argument for the sake of completeness.
Iterating inequality \eqref{ingegality1} once, we have
\begin{align*}
\phi(t) &\le A + \frac{L}{\Gamma(\theta)} \int_0^t (t-s)^{\theta-1} \phi(s) \, ds\\
&\le  A + \frac{L}{\Gamma(\theta)}  \int_0^t (t-s)^{\theta-1} \left( A
  + \frac{L}{\Gamma(\theta)} \int_0^s (s-r)^{\theta-1} \phi(r) \, dr
\right)  ds\\  
&= A\(1 + \frac{L}{\theta \Gamma(\theta)} T^{\theta} \) +
\frac{L^2}{\Gamma(\theta)^2} \int_0^t (t-s)^{\theta-1}\int_0^s
(s-r)^{\theta-1} \phi(r) \, dr \, ds. 
\end{align*}
From Fubini's Theorem, we get
\begin{align*}
\int_0^t (t-s)^{\theta-1}\int_0^s (s-r)^{\theta-1} \phi(r) \, dr \, ds
&= \int_0^t \phi(r) \int_r^t (t-s)^{\theta-1} (s-r)^{\theta-1} ds \,
dr \\ 
&= \int_0^t \phi(r) (t-r)^{2\theta-1} \left(\int_0^1 (1 -
  \tau)^{\theta-1} \tau^{\theta-1} d\tau \right)  dr \\ 
&= \beta(\theta,\theta) \int_0^t \phi(r) (t-r)^{2\theta-1}  dr,
\end{align*}
where $\beta$ is the beta function. Therefore, we have
\begin{equation}
\phi(t) \le  C_T(\theta) A+ \frac{L^2}{\Gamma(\theta)^2}
\beta(\theta,\theta) \int_0^t \phi(s) (t-s)^{2\theta-1} \, ds. 
\label{inegalite2}
\end{equation}
Iterating the estimate \eqref{inegalite2} $n$ times, 
with $n\theta\ge 1$, we get the
following estimate:
\begin{equation*}
\phi(t) \le \tilde{C}_T(\theta) A + \tilde{L}_T(\theta) \int_0^t
\phi(s) (t-s)^\alpha ds, 
\end{equation*}
with $\alpha \ge 0$, and where $\tilde{L}_T(\theta)$
is a positive constant which depends on $T$ and $\theta$. The lemma
then follows from the classical Gronwall Lemma.
\end{proof}

%\begin{remark}
%Let $\theta = \frac{1}{4}$ then if 
%\begin{equation*}
%\phi(t) \le A + L\int_0^t (t-s)^{-3/4} \phi(s) \, ds, 
%\end{equation*}
%we have from Fubini's Theorem
%\begin{align*}
%\phi(t) &\le A + L\int_0^t (t-s)^{-3/4} \left(A + L \int_0^s
%  (s-r)^{-3/4} \phi(r) \, dr\right) \, ds \\ 
%&\le A(1+4 T^{1/4} L) + L^2 \beta(\frac{1}{4}, \frac{1}{4}) \int_0^t
%(t-s)^{-1/2} \phi(s) \, ds \\ 
%&\le  A(1+4 T^{1/4} L)+   L^2 \beta(\frac{1}{4}, \frac{1}{4}) A(1+4
%T^{1/4} L) 2 T^{1/2} + L^4 \beta(\frac{1}{4}, \frac{1}{4})^2 \pi
%\int_0^t \phi(s) \, ds \\ 
%&\le   A(1+4 T^{1/4} L)\left(1+ L^2 \beta(\frac{1}{4}, \frac{1}{4}) 2 T^{1/2}\right) + L^4 \beta(\frac{1}{4}, \frac{1}{4})^2 \pi \int_0^t \phi(s) \, ds, 
%\end{align*}
%because $ \beta \left( \frac{1}{2}, \frac{1}{2} \right) = \pi$. \\
%Finally, from the classical Gronwall Lemma, we obtain
%\begin{equation}
%\phi(t) \le  A\(1+4 T^{1/4} L\)\left(1+ L^2 \beta(\frac{1}{4},
%  \frac{1}{4}) 2 T^{1/2}\right)e^{ L^4 \beta(\frac{1}{4},
%  \frac{1}{4})^2 \pi t }.  
%\end{equation}
%\end{remark}

\begin{lemma}[Modified fractional Gronwall Lemma] 
\label{lemme2}
Let $\phi:[0,T] \to \R_+$ be a bounded measurable function and
$P$ be a polynomial with positive coefficients and no constant
term. We assume there exists two positive constants $C$ and $\theta
\in ]0,1[$ such that for all $t \in [0,T]$, 
\begin{equation}
0 \le \phi(t) \le \phi(0) + P(t) + C \frac{d^{-\theta}}{d t^{-\theta}}\phi(t) .
\label{inegalite3}
\end{equation}
Then there exists $C_T(\theta)$ such that for all $t\in [0,T]$,
\begin{equation*}
\phi(t) \le C_T(\theta)\,\phi(0) + C_T(\theta) \, P(t).
\end{equation*}
\end{lemma}

\begin{proof}
Arguing as in the proof of Lemma~\ref{lemme1}, we iterate the previous
inequality. After one iteration, we get
\begin{align*}
\phi(t) &\le \tilde C \phi(0) + P(t) +\frac{C}{\Gamma(\theta)}\int_0^t
(t-s)^{\theta-1}P(s)ds \\
&\quad + \frac{C^2}{\Gamma(\theta)^2}\int_0^t
(t-s)^{\theta-1}\(\int_0^s (s-r)^{\theta-1}\phi(r)dr\)ds\\
&\le \tilde C \phi(0) + \tilde C P(t) + \frac{C^2}{\Gamma(\theta)^2}
\beta(\theta,\theta) \int_0^t \phi(s)(t-s)^{2\theta-1}ds,
\end{align*}
where we have used the assumptions on $P$, and Fubini's Theorem again
for the last term. Iterating sufficiently many times, we infer like in
the proof of Lemma~\ref{lemme1}:
\begin{equation}\label{eq:1717}
  \phi(t)\le c_0 \phi(0)+c_0 P(t)+ \underline C \int_0^t
  \phi(s)(t-s)^\alpha ds,
\end{equation}
with $\alpha> 0$.
Set 
\begin{equation*}
\psi(t) =\left( c_0 \phi(0)+c_0 P(t)+ \underline C \int_0^t
  \phi(s)(t-s)^\alpha ds\right) e^{- C_1  t}.
\end{equation*}
Then 
\begin{align*}
\psi'(t) &= \Big(  c_0 P'(t) + \underline C \alpha \int_0^t
  \phi(s)(t-s)^{\alpha-1} ds \\
&\quad - C_1 \Big( c_0 \phi(0) +  c_0
    P(t) + \underline C \int_0^t \phi(s)(t-s)^\alpha ds  \Big)  \Big)
e^{-C_1 t}.
\end{align*}
Using \eqref{eq:1717} to control the second term, and choosing $C_1$
sufficiently large, we infer:
\begin{equation*}
  \psi'(t) \le c_0 P'(t)e^{-C_1 t}.
\end{equation*}
Since $P(0) = 0, $ $\psi(0)= c_0 \phi(0)$, for all $t \in [0,T]$,
\begin{align*}
\phi(t) &\le \psi(t)e^{C_1 t} \le c_0 \phi(0)e^{ C_1 t }
+ c_0\int_0^t   P'(s) e^{C_1 (t-s)} ds 
 \\
 &\le  c_0 \phi(0)e^{ C_1 T } +c_0 e^{ C_1 T } 
 \int_0^t P'(s) \, ds \le  c_0 e^{ C_1 T }\left( \phi(0) + P(t)\right) .
\end{align*}
This completes the proof.
\end{proof}

\section{Estimates on the various flows\label{sectionlemme}}

%In this section, we begin by showing two Gronwall type lemmas adapted
%to the case of fractional equations. We will see that these lemmas are
%based on the resolution of an inequation involving a Riemann--Liouville
%operator.  
% We next collect some estimates on $X^t$, $Y^t$, $Z_L^t$ and $S^t $.  
 % lemma based on the resolution of an inequation involving a Riemann-Liouville operator. 
 %In the next sub-sections, we collect some estimates on $X^t$, $Y^t$ and $S^t $.   

\subsection{Estimates on linear flows}
In this paragraph, we collect several estimates concerning the
convolutions with $D$, $K$ and $G$, which will be useful in the
estimates of the local error of the scheme. 
\begin{proposition}\label{estimnonlocal}
Let $s\in \R$ and $\varphi \in  H^s(\R)$. Then $\I[\varphi] \in
H^{s-4/3}(\R)$ and we have  
\begin{equation}
\|\I[\varphi]\|_{H^{s-4/3}(\R)} \le 4 \pi^2 \Gamma\(\frac{2}{3}\)
\|\varphi\|_{H^s(\R)}. 
\end{equation}
\end{proposition}

\begin{proof}
For all $s\in \R$ and all $\varphi \in  H^s(\R)$, we have, using \eqref{fourier}
\begin{align*}
\|\I[\varphi]&\|_{H^{s-4/3}(\R)} = \left( \int_\R (1 +
  |\xi|^2)^{s-4/3} |\F(\I[\varphi])(\xi)|^2 \, d\xi   \right)^{1/2} \\ 
&= 4 \pi^2 \Gamma\(\frac{2}{3}\)  \left( \int_\R (1 + |\xi|^2)^{s-4/3}
  \left|\frac{1}{2}-i \mbox{ sgn}(\xi) \frac{\sqrt{3}}{2}\right|^2 |\xi |^{8/3}
  |\F(\varphi)(\xi)|^2 \, d\xi \right)^{1/2}  \\ 
&= 4 \pi^2 \Gamma\(\frac{2}{3}\) \left( \int_\R \left( \frac{|\xi|^2}{
      1+ |\xi|^2}\right) ^{4/3} (1 + |\xi|^2)^s | \F(\varphi)(\xi)|^2
  \, d\xi \right)^{1/2} \\ 
&\le  4 \pi^2 \Gamma\(\frac{2}{3}\) \left[\int_\R (1 + |\xi|^2)^s |
  \F(\varphi)(\xi)|^2 \, d\xi \right]^{1/2} = 4 \pi^2
\Gamma\(\frac{2}{3}\) \|\varphi\|_{H^s(\R)},
\end{align*}
hence the result. 
\end{proof}

\begin{lemma} 
\label{lemme3} 
$(1)$ Let $n \in \N.$ Then, for all $v \in H^{n}(\R)$ and all $t>0$,
\begin{equation*}
\|X^t v\|_{H^{n}(\R)} \le e^{\beta_0 t} \|v\|_{H^n(\R)},
\end{equation*}
where $\beta_0 = - \min \RE(\phi_\I) > 0$. \\
$(2)$ Let $n \in \N.$ There exists $C$ such that for all $v \in H^2(\R)$ and all $t>0$,
\begin{equation}
\|X^t v - v \|_{H^n(\R)} \le C \, t \, e^{\beta_0 t} \|v \|_{H^{n+2}(\R)}.
\end{equation}
\end{lemma}

\begin{proof}
Using Plancherel formula, we have
\begin{align*}
\|X^t v \|_{L^2(\R)}^2 &= \|D(t,\cdot) \ast v \|_{L^2(\R)}^2 \\
&= \|\F \left( D(t,\cdot)\right)  \F v \|_{L^2(\R)}^2 = \int_\R |\F
\left( D(t,\cdot)\right) (\xi)|^2 |\F v(\xi)|^2 \, d\xi \\ 
&=  \int_\R e^{-2t \phi_\I(\xi)}  |\F v(\xi)|^2 \, d\xi \le e^{2
  \beta_0 t} \|v\|_{L^2(\R)}^2. 
\end{align*}
Moreover, since
\begin{equation*}
\partial^{n}_{x} X^t v = D(t, \cdot) \ast \partial^{n}_{x} v 
\end{equation*}
then, from again Plancherel formula, we have
\begin{equation*}
\|\partial^n_{x} X^t v\|_{L^2(\R)} \le e^{\beta_0 t} \| \partial^n_{x} v\|_{L^2(\R)},
\end{equation*}
hence the first point of the lemma.

Let $n\in \N$, $v \in H^{n+2}(\R)$. We have
\begin{equation*}
\|X^t v - v \|_{H^n(\R)} = \left\| \int_{0}^t \stackrel{.}{X^s} v \,
  ds \right\|_{H^n(\R)}.  
\end{equation*}
But from the definition of $X^t$, $\stackrel{.}{X^s}$ is given by
\begin{equation*}
\stackrel{.}{X^s} v = \eta \d_x^2 X^s v - \I[X^s v] =  \eta X^s \d_x^2
v - \I[X^s v],  
\end{equation*}
since $X^s \d_x^2 v = D(s, \cdot) \ast \d_x^2 v = \d_x^2 \left( D(s,
  \cdot) \ast v\right)$. Thus, 
using Proposition~\ref{estimnonlocal} and the first point of this lemma, we get
\begin{align*}
\|X^t v - v \|_{H^n(\R)} & \le  \eta \int_{0}^t \| X^s \d_x^2 v
\|_{H^n(\R)} \, ds +  \int_{0}^t \|\I[X^s v ] \|_{H^n(\R)} \, ds \\ 
& \le   \eta \, t \, e^{\beta_0 t}\|v\|_{H^{n+2}(\R)} + \int_0^t \|\I[X^s
v ] \|_{H^n(\R)}ds \\ 
& \le \eta \, t \, e^{\beta_0 t}  \|v\|_{H^{n+2}(\R)} + 4 \pi^2
\Gamma\(\frac{2}{3}\) \int_0^t \|X^s v \|_{H^{n+4/3}(\R)}ds \\ 
& \le \eta \, t \, e^{\beta_0 t}  \|v\|_{H^{n+2}(\R)} + 4 \pi^2
\Gamma\(\frac{2}{3}\) \int_0^t \|X^s v \|_{H^{n+2}(\R)}ds \\ 
& \le  \left( \eta + 4 \pi^2 \Gamma\(\frac{2}{3}\) \right) t \,
e^{\beta_0 t}  \|v\|_{H^{n+2}(\R)},    
\end{align*}
hence the result.
\end{proof}
Recalling that $K$ corresponds to $D$ in the case $\eta=1$, we readily
infer:
\begin{corollary}\label{corokernel}
For all $w \in H^2(\R)$ and all $t>0$,
\begin{equation}
\|K(t, \cdot)\ast w - w \|_{L^2(\R)} \le C\, t \, e^{\alpha_0 t} \|w\|_{H^2(\R)},
\end{equation}
where $C$ is a positive constant independent of $t$ and $w$.
\end{corollary}

We conclude this paragraph with an analogous result on the heat kernel $G$:
\begin{lemma}\label{lemmechaleur}
For all $w \in H^2 (\R)$ and all $t > 0$,
\begin{equation*}
\|G(t, \cdot) \ast w - w\|_{L^2(\R)} \le \varepsilon\, t \, \|w\|_{H^2(\R)}.
\end{equation*}
\end{lemma}

\begin{proof}
 % Let us denote $G^t$ the solution of the following heat equation
%\begin{equation*}
%u_t - \varepsilon u_{xx} = 0.
%\end{equation*} 
Proceeding as above, we have:
\begin{align*}
G(t, \cdot)\ast w - w &= \int_0^t \frac{\partial}{\partial t}\left(
  G(s, \cdot)\ast w\right) \, ds =  \varepsilon \int_0^t
\d_x^2\left(G(s, \cdot)\ast w \right) \, ds  \\
&=  \varepsilon \int_0^t  G(s, \cdot)\ast\d_x^2 w \, ds.
\end{align*}
%We thus have
%\begin{equation*}
%G(t, \cdot) \ast w - w = \int_0^t \frac{d \, G^s w}{ds} \, ds = \int_0^t G^s \Delta w \, ds, 
%\end{equation*}
Taking the norm $L^2$ and using Proposition~\ref{heat}, Young's
inequality yields
\begin{equation*}
\|G(t, \cdot) \ast w - w\|_{L^2(\R)} \le   \varepsilon \int_{0}^t
\|G(s,\cdot)\|_{L^1(\R)} \|\d_x^2 w \|_{L^2(\R)} \, ds\le
\varepsilon \, t \, \|w\|_{H^2(\R)},
\end{equation*}
hence the result.
\end{proof}

\subsection{Estimates on $Y^t$}

We now turn to  the viscous Burgers' equation \eqref{burgers}:
\begin{equation}
  \label{eq:vb}
   \d_t w -\varepsilon  \d_x^2w +w \d_x w=0;\quad w_{\mid t=0}=w_0.  
\end{equation}

\begin{remark}[Hopf--Cole transform]
 It is well-known that the change of unknown function
  \begin{equation*}
    w = -2\varepsilon \frac{1}{\phi}\d_x \phi = -2\varepsilon \d_x\(\ln \phi\),
  \end{equation*}
turns the viscous Burgers' equation into the heat equation:
\begin{equation*}
  \d_t \phi -\varepsilon \d_x^2\phi=0. 
\end{equation*}
We infer the explicit formula:
\begin{equation*}
  w(t,x) = -2\varepsilon \d_x \ln \(\frac{1}{\sqrt{4\pi\varepsilon
      t}}\int_{-\infty}^{+\infty} \exp\(-\frac{(x-y)^2}{4\varepsilon
    t}-\frac{1}{2\varepsilon} \int_0^y w_0(z)dz\)dy\).
\end{equation*}
However, this formula does not seem very helpful in order to establish
Proposition~\ref{prop:vb}. 
\end{remark}

\begin{proposition}\label{prop:vb}
  Let $w_0\in H^1(\R)$. Then \eqref{eq:vb} has a unique solution $w\in
  C(\R_+;H^1(\R))$. In addition, there exists $C=C(\varepsilon,\|w_0\|_{H^1(\R)})$
  such that for all $t\ge 0$,
  \begin{equation*}
    \|w(t)\|_{L^2(\R)}\le \|w_0\|_{L^2(\R)},\quad \|\d_x w(t)\|_{L^2(\R)}\le
    \|w_0'\|_{L^2(\R)} e^{C(t^{5/8}+t)}.
  \end{equation*}
If in addition $w_0\in H^n(\R)$, for some $n\ge 2$, then $w\in
C(\R_+;H^n(\R))$ and for 
all $T>0$, there
exists $M=M(\varepsilon,T,\|w_0\|_{H^2(\R)})$ 
  such that for all $t\in [0,T]$,
  \begin{equation*}
    \|w(t)\|_{H^n(\R)}\le
    \|w_0\|_{H^n(\R)} e^{Mt}.
  \end{equation*}
\end{proposition}

\begin{proof}
  The existence and uniqueness part being standard, we focus on the
  estimates. The $L^2$ estimate yields (formally, multiply
  \eqref{eq:vb} by $w$  and integrate)
  \begin{equation}\label{eq:L2}
    \frac{1}{2}\frac{d}{dt}\|w(t)\|_{L^2}^2 + \varepsilon \|\d_x
    w(t)\|_{L^2}^2=0, 
  \end{equation}
and the $H^1$ estimate (differentiate \eqref{eq:vb} with respect to
$x$, multiply by $\d_x w$ and integrate),
\begin{equation}\label{eq:H1}
  \frac{1}{2}\frac{d}{dt}\|\d_x w(t)\|_{L^2}^2 + \varepsilon\|\d_x^2
  w(t)\|_{L^2}^2=-\frac{1}{4}\int_{\R}\(\d_x w(t,x)\)^3 dx. 
\end{equation}
The $L^2$ estimate \eqref{eq:L2} shows that the map $t\mapsto
\|w(t)\|_{L^2}^2$ is non-increasing:
\begin{equation*}
  \|w(t)\|_{L^2}\le \|w_0\|_{L^2},\quad \forall t\ge 0.
\end{equation*}
An integration by parts and Cauchy--Schwarz inequality then yield
\begin{equation}\label{eq:cachan92}
  \|\d_x w(t)\|_{L^2}^2\le \|w(t)\|_{L^2}\|\d_x^2 w(t)\|_{L^2}\le
  \|w_0\|_{L^2}\|\d_x^2 w(t)\|_{L^2}. 
\end{equation}
In order to take advantage of the smoothing effect provided by the
viscous part, integrate 
\eqref{eq:H1} in time and write
\begin{equation*}
  \varepsilon \int_0^t \|\d_x^2 w(s)\|_{L^2}^2 ds\le \frac{1}{2}\|w_0'\|_{L^2}^2 +
  \frac{1}{4}\int_0^t \|\d_x w(s)\|_{L^3}^3 ds. 
\end{equation*}
Gagliardo--Nirenberg inequality yields
\begin{equation*}
  \|\d_x w\|_{L^3} \le C \|\d_x w\|_{L^2}^{5/6}\|\d_x^2 w\|_{L^2}^{1/6},
\end{equation*}
so using \eqref{eq:cachan92}, we infer:
\begin{align*}
  \varepsilon\int_0^t \|\d_x^2 w(s)\|_{L^2}^2 ds & \le
  \frac{1}{2}\|w_0'\|_{L^2}^2 +C  \int_0^t \|\d_x
  w(s)\|_{L^2}^{5/2}\|\d_x^2 w(s)\|_{L^2}^{1/2} ds \\
&\le
  \frac{1}{2}\|w_0'\|_{L^2}^2 + C\|w_0\|_{L^2}^{5/4}\int_0^t \|\d_x^2
  w(s)\|_{L^2}^{7/4} ds. 
\end{align*}
In view of H\"older inequality in the last integral in time, 
\begin{align*}
 \varepsilon \int_0^t \|\d_x^2 w(s)\|_{L^2}^2 ds &\le
  \frac{1}{2}\|w_0'\|_{L^2}^2 + C\|w_0\|_{L^2}^{5/4}\(\int_0^t \|\d_x^2
  w(s)\|_{L^2}^2 ds\)^{7/8}t^{1/8}\\
&\le \frac{1}{2}\|w_0'\|_{L^2}^2 + \frac{\varepsilon}{2} \int_0^t
\|\d_x^2 w(s)\|_{L^2}^2 ds + C\(\|w_0\|_{L^2}\) t,
\end{align*}
where we have used Young inequality $ab\lesssim a^{8/7}+b^8$. We infer 
\begin{equation}\label{eq:L2H2}
  \varepsilon\int_0^t \|\d_x^2 w(s)\|_{L^2}^2 ds \le \|w_0'\|_{L^2}^2  +
  C\(\|w_0\|_{L^2}\) t. 
\end{equation}
Gagliardo--Nirenberg inequality $\|f\|_{L^\infty}\le \sqrt 2
\|f\|_{L^2}^{1/2}\|f'\|_{L^2}^{1/2}$ now yields
\begin{align*}
  \int_0^t\|\d_x w(s)\|_{L^\infty} ds &\le \sqrt 2 \int_0^t
  \|\d_x w(s)\|_{L^2}^{1/2} \|\d_x^2 w(s)\|_{L^2}^{1/2} ds\\
&\le C\(\varepsilon,\|w_0\|_{L^2}\)
\int_0^t\|\d_x^2 w(s)\|_{L^2}^{3/4} ds \\
&\le C\(\varepsilon,\|w_0\|_{L^2}\)
\(\int_0^t\|\d_x^2 w(s)\|_{L^2}^{2} ds\)^{3/8} t^{5/8} \\
&\le C\(\varepsilon,\|w_0\|_{H^1}\)\(1+t\)^{3/8}t^{5/8}\le
C\(\varepsilon,\|w_0\|_{H^1}\)\(t^{5/8} +t\), 
\end{align*}
where we have used \eqref{eq:cachan92}, H\"older inequality and
\eqref{eq:L2H2}, successively. 

Integrate the $H^1$ estimate \eqref{eq:H1} with respect to time,
and now discard the viscous part whose contribution is non-negative:
\begin{equation}
  \label{eq:ineqH1}
  \begin{aligned}
  \|\d_x w(t)\|_{L^2}^2 &\le \|w_0'\|_{L^2}^2 + \frac{1}{2}\int_0^t
  \|\d_x w(s)\|_{L^3}^3 ds \\
&\le \|w_0'\|_{L^2}^2 + \frac{1}{2}\int_0^t
  \|\d_x w(s)\|_{L^\infty} \|\d_x w(s)\|_{L^2}^2ds.
\end{aligned}
\end{equation}
The first part of the proposition then follows from the 
Gronwall lemma.  

To complete the proof of the proposition, we use the general $H^n$
estimate, for $n\in \N$: set $\Lambda = (1-\d_x^2)^{1/2}$. Applying $\Lambda^n$ to
\eqref{eq:vb} yields
\begin{equation*}
  \frac{1}{2}\frac{d}{dt}\|\Lambda^n w(t)\|_{L^2}= -\int_\R \Lambda^n
  w(t,x)\Lambda^n\(w\d_x w\)(t,x)dx +\eps\int_\R \Lambda^n
  w(t,x)\Lambda^n\d_x^2 w(t,x)dx.
\end{equation*}
Integrating by parts, the last term is non-positive, since
\begin{equation*}
  \int_\R \Lambda^n
  w(t,x)\Lambda^n\d_x^2 w(t,x)dx =- \int_\R \left|\Lambda^n\d_x
  w(t,x)\right|^2dx.
\end{equation*}
Write 
\begin{align*}
  \int_\R \Lambda^n
  w(t,x)\Lambda^n\(w\d_x w\)(t,x)dx &= \int_\R \Lambda^n
  w(t,x)\(w\d_x \Lambda^n w\)(t,x)dx \\
&+\int_\R \Lambda^n
  w(t,x)\( w\d_x \Lambda^n w- \Lambda^n\(w\d_x w\)\)(t,x)dx.
\end{align*}
Integrating by parts the first term yields
\begin{align*}
  \left|\int_\R \Lambda^n
  w(t,x)\(w\d_x \Lambda^n w\)(t,x)dx\right|
&=\frac{1}{2}\left|\int_\R \d_x\( \Lambda^n 
  w(t,x)\)^2w(t,x)dx\right|\\
& =\frac{1}{2}\left|\int_\R \( \Lambda^n
  w(t,x)\)^2\d_x w(t,x)dx\right| \\
&\le \frac{\|\d_x
  w(t)\|_{L\infty}}{2}\|w(t)\|_{H^n}^2. 
\end{align*}
In view of Kato-Ponce estimate \cite{KaPo88}
\begin{equation}\label{eq:KatoPonce}
  \| \Lambda^{n}(fg)- f\Lambda^{n}g\|_{L^2}\le C \|\d_x
  f\|_{L^\infty} \|g\|_{H^{n-1}}+ \|f\|_{H^{n}}\|g\|_{L^\infty},
\end{equation}
we have (with $f=w$ and $g=\d_x w$)
\begin{align*}
\|w\d_x \Lambda^n w- \Lambda^n\(w\d_x w\)\|_{L^2}   \le C \|\d_x
w\|_{L^\infty} \|w\|_{H^n}.
\end{align*}
Leaving out the viscous term, Gronwall lemma yields the \emph{a priori} estimate
\begin{equation}
  \label{eq:aprioriHn}
  \|w(t)\|_{H^n}\le \|w_0\|_{H^n}\exp{\(
    C\int_0^t\|\d_xw(s)\|_{L^\infty}ds\)},
\end{equation}
where $C$ depends only on $n\in \N$. 
In particular, for $n=2$,
Gronwall lemma implies $\|\d_x^2 w(t)\|_{L^2}\le \|w_0\|_{H^2}
e^{C(t^{5/8}+t)}$, where $C=C(\varepsilon,\|w_0\|_{H^1})$. We bootstrap, thanks to
Gagliardo--Nirenberg inequality again:  
\begin{equation*}
  \|\d_x w(t)\|_{L^\infty} \le \sqrt 2 
  \|\d_x w(t)\|_{L^2}^{1/2} \|\d_x^2 w(t)\|_{L^2}^{1/2} \le \sqrt
  2\|w_0\|_{H^2}  e^{C(t^{5/8}+t)}. 
\end{equation*}
Therefore, for $t\in [0,T]$,
\begin{equation*}
  \int_0^t\|\d_x w(s)\|_{L^\infty} ds \le \sqrt
  2 \|w_0\|_{H^2} \times t \times e^{C(T^{5/8}+T)}. 
\end{equation*}
The last estimates of the proposition then follow from \eqref{eq:aprioriHn}.
\end{proof}

\begin{lemma}\label{estimnl1}
Let  $T>0$. For all $w \in H^1(\R)$,
there exists $C=C(T,\|w\|_{L^2(\R)})$ such that
\begin{equation}
\|Y^t w \|_{H^1(\R)} \le e^{Ct}\|w\|_{H^1(\R)},\quad\forall t\in [0,T].
\label{estmnle1}
\end{equation}
\end{lemma}

\begin{proof}
Differentiating the Duhamel formula \eqref{burgerduhamel} in space, we have
\begin{equation*}
\partial_x Y^t w = G(t, \cdot) \ast \partial_x w - \int_0^t \partial_x
G(t-s, \cdot) \ast (Y^s w) \partial_x (Y^s w) \, ds. 
\end{equation*}
Using Young inequality and inequality \eqref{estimapriori1}, we infer: 
\begin{align*}
  \|\d_x Y^t w \|_{L^2(\R)}& \le \|\partial_x w\|_{L^2(\R)} \\
&\quad +
  \int_0^t \| \partial_x G(t-s, \cdot)\|_{L^2(\R)} \|(Y^s w) \partial_x
  (Y^s w)\|_{L^1(\R)} \, ds.  
\end{align*}
In view of Proposition~\ref{heat}, this implies:
\begin{align*}
  \|\d_x Y^t w \|_{L^2(\R)}& \le \|\partial_x w\|_{L^2(\R)} \\
&\quad +
  C_0\int_0^t (t-s)^{-3/4}\|(Y^s w) \partial_x
  (Y^s w)\|_{L^1(\R)} \, ds.  
\end{align*}
Writing
\begin{align*}
  \|(Y^s w) \partial_x
  (Y^s w)\|_{L^1}\le \|Y^s w\|_{L^2}\|\d_x(Y^s
w)\|_{L^2}\le \|w\|_{L^2}\|\d_x(Y^s
w)\|_{L^2}, 
\end{align*}
and invoking the fractional Gronwall Lemma~\ref{lemme1} with $\theta
=1/4$, the lemma follows.  
\end{proof}

\begin{corollary}\label{cor:condest}
  Let $n\ge 1$ and $w\in H^n(\R)$.  Let $T,\alpha>0$. If 
  \begin{equation*}
    \|w\|_{L^2(\R)}\le \alpha,
  \end{equation*}
then there exists $c$ depending only on $T$ and $\alpha$ such that
\begin{equation*}
  \|Y^t w\|_{H^n(\R)}\le e^{c t}\|w\|_{H^n(\R)},\quad \forall t\in [0,T]. 
\end{equation*}
\end{corollary}
\begin{proof}
  Denote $w(t)=Y^t w$. From Lemma~\ref{estimnl1}, 
  \begin{equation*}
    \|w(t)\|_{H^1(\R)}\le \|w_0\|_{H^1}e^{C t},
  \end{equation*}
where $C$ depends only on $\alpha$ and $T$. 
From the proof of the first part of
  Proposition~\ref{prop:vb} (see 
  the estimate after \eqref{eq:L2H2}), we
  infer
  \begin{equation*}
    \int_0^t\|\d_x  w(s)\|_{L^\infty} ds \le C(\alpha,T),\quad \forall
    t\in [0,T].  
  \end{equation*}
The corollary then stems from \eqref{eq:aprioriHn}. 
\end{proof}

\subsection{Estimates on the splitting operator $Z_L^t$}
Combining the estimates on $X^t$ and $Y^t$ established in the previous
two sections, we infer:
\begin{corollary}\label{cor:split1}
\begin{enumerate}
\item 
For all $u \in L^{2}(\R)$ and all $t>0$,
\begin{equation*}
\|Z_L^t \, u \|_{L^{2}(\R)} \le e^{\beta_0 t} \, \|u\|_{L^{2}(\R)},
\end{equation*}
where $\beta_0 = -\min \RE (\phi_\I)>0$.
\item Let $T>0, n \in \N^*$ and $u \in H^n(\R)$. There exists
  $C=C(T,\|u\|_{L^2(\R)})$ such that for all $t \in [0,T]$,
\begin{equation*}
\|Z^t_L \, u\|_{H^n(\R)} \le e^{Ct}\|u\|_{H^n(\R)}.
\end{equation*}
\end{enumerate}
\end{corollary}

\begin{proof}
The first point is a direct consequence of the relation
\eqref{estimapriori1} and Lemma \ref{lemme3}.\\
The second point is readily established with Lemma~\ref{lemme3} and
Corollary~\ref{cor:condest}. 
\end{proof}

\subsection{Estimates on the exact flow $S^t$}

\begin{lemma}[$L^2$-a priori estimate]\label{estimapriori2}
Let $u_0 \in L^{2}(\R)$ and $T>0$. Then, the unique mild solution $u
\in C([0,T]; L^{2}(\R))\cap C(]0,T]; H^{2}(\R))$ of \eqref{fowlereqn5}
satisfies, for all $t \in [0,T]$  
\begin{equation*}
\|u(t,\cdot) \|_{L^2(\R)} \le e^{\alpha_0 t} \|u_0\|_{L^2(\R)},
\end{equation*}
where $\alpha_0 = -\min \RE(\psi_\I)>0$. 
\end{lemma}

\begin{proof}
Multiplying  \eqref{fowlereqn5} by
$u$ and integrating with respect to the space variable, we get:
\begin{equation*} 
\int_{\R} u_{t}u \: dx+ \int_{\R} \left(\I[u]-u_{xx}\right)u \: dx =0
\end{equation*}
because the nonlinear term is zero. 
Using \eqref{fourier} and the fact that $u$ and $\int_\R
(\I[u]-\partial_{xx}^2 u) u \, dx$ are real, we get 
\begin{equation*} 
\int_\R (\I[u]-\partial_{xx}^2 u) u \, dx = \int_\R \F^{-1}(\psi_\I
\F u) u \, dx = \int_\R \psi_\I |\F u|^2 \, d\xi
 = \int_\R \mbox{Re}(\psi_\I) |\F u|^2 \, d\xi.
\end{equation*} 
We infer
\begin{eqnarray*}
\frac{1}{2} \frac{d}{dt}\| u(t, \cdot) \|_{L^2}^{2} & \le &
\alpha_{0} \| u(t) \|_{L^{2}}^{2}  
\end{eqnarray*}
where $\alpha_0=-\min \RE (\psi_{\I})>0$. The result then follows from
the Gronwall lemma. 
\end{proof}

\begin{lemma}\label{lemme8}
 Let $n \in \N^*, $ $u_0 \in H^n(\R)$ and
  $T>0$. There exists $C_T(\|u_0\|_{H^{n-1}(\R)})$ such that  the
  unique mild solution $u \in C([0,T]; H^{n}(\R))$ satisfies  
\begin{equation}
\|u(t, \cdot)\|_{H^n(\R)} \le C_T(\|u_0\|_{H^{n-1}(\R)}) \|u_0\|_{H^n(\R)}.
\label{estimfowler}
\end{equation}
\end{lemma}

\begin{proof}
The proof is similar to the one given in Lemma \ref{estimnl1}.
Differentiating the Duhamel formula \eqref{duhamel5} in space, we have
\begin{equation*}
\partial_x u = K(t, \cdot) \ast \partial_x u_0 - \int_0^t \partial_x
K(t-s, \cdot) \ast \(u(s) \partial_x u(s)\) \, ds. 
\end{equation*}
Using Young inequality and Proposition~\ref{kernel2}, we infer,
for any integer $n\ge 1$: 
\begin{align*}
  \|\d_x u(t) \|_{H^{n-1}(\R)}& \le e^{\alpha_0 t}\|\partial_x u_0\|_{H^{n-1}(\R)} \\
&\quad + 
  \int_0^t \| \partial_x K(t-s, \cdot)\|_{L^2(\R)} \|u(s) \partial_x
  u(s)\|_{W^{n-1,1}(\R)} \, ds.  
\end{align*}
In view of Proposition~\ref{kernel2}, this implies:
\begin{align*}
  \|\d_x u(t) \|_{H^{n-1}(\R)}& \le e^{\alpha_0 t}\|\partial_x u_0\|_{H^{n-1}(\R)} \\
&\quad +
  C_0\int_0^t (t-s)^{-3/4}\|u(s) \partial_x
  u(s)\|_{W^{n-1,1}(\R)} \, ds.  
\end{align*}
For $n=1$, we use Lemma~\ref{estimapriori2} to have
\begin{align*}
  \|u(s) \partial_x
  u(s)\|_{L^1}\le \|u(s)\|_{L^2}\|\d_x u(s)\|_{L^2}\le e^{\alpha_0 T}
  \|u_0\|_{L^2}\|\d_x u(s)\|_{L^2}. 
\end{align*}
The fractional Gronwall Lemma~\ref{lemme1} with $\theta =1/4$ then yields
\begin{equation*}
\|\partial_x u(t)\|_{L^{2}(\R)} \le  e^{Ct+CT}
\|\partial_x u_0\|_{L^{2}(\R)} ,
\end{equation*}
where $C$ depends only on $T$ and $\|u_0\|_{L^2}$. 
From \eqref{estimapriori1}, this implies the lemma in the case $n=1$. 
For $n\ge 1$, Leibniz rule and Cauchy--Schwarz inequality yield
\begin{equation*}
  \|u(s) \partial_x
  u(s)\|_{W^{n-1,1}(\R)} \le C(n)\|u(s)\|_{H^{n-1}(\R)}\|\partial_x
  u(s)\|_{H^{n-1}(\R)}.
\end{equation*}
The lemma then easily follows by induction on $n$. 
\end{proof}

We will also need the fact that the flow map $S^t$ is
uniformly Lipschitzean on balls of $H^2(\R)$.
\begin{proposition}\label{prop:lipschitz}
  Let $T,R>0$. There exists $K=K(R,T)<\infty$ such that if
  \begin{equation*}
    \|u_0\|_{H^2(\R)} \le R, \, \|v_0\|_{H^2(\R)} \le R,
  \end{equation*}
then 
\begin{equation*}
  \|S^t u_0-S^tv_0\|_{L^2(\R)}\le K\|u_0-v_0\|_{L^2(\R)},\quad \forall
  t\in [0,T].
\end{equation*}
\end{proposition}
\begin{proof}
  Set $u(t)=S^tu_0$, $v(t)=S^t v_0$ and $w=u-v$. It solves
  \begin{equation}\label{eq:lip}
    \d_t w +\I[w]-\partial_{xx}^2 w = v\d_x v-u\d_x u = -u\d_x w -
    w\d_x v. 
  \end{equation}
The $L^2$ energy estimate yields:
\begin{equation*}
  \frac{1}{2}\frac{d}{dt}\|w\|_{L^2}^2 +\int_{\R}w \(u\d_x w +w\d_x
  v\)\le \alpha_0 \|w(t)\|_{L^2}^2,
\end{equation*}
where the  term $\int w\(\I[w]-\partial_{xx}^2 w\)$ has been estimated
as in the proof of Lemma~\ref{estimapriori2}. We infer
\begin{align*}
  \frac{1}{2}\frac{d}{dt}\|w(t)\|_{L^2}^2 &\le \(\alpha_0+
  \frac{1}{2}\|\d_xu(t) \|_{L^\infty} +\|\d_x
  v(t)\|_{L^\infty}\)\|w(t)\|_{L^2}^2\\
&\le C\(1+\|u(t)\|_{H^2}+\|v(t)\|_{H^2}\) \|w(t)\|_{L^2}^2\le C(R,T) \|w(t)\|_{L^2}^2,
\end{align*}
where we have used Sobolev embedding and Lemma~\ref{lemme8}. Gronwall
lemma yields the result. 
\end{proof}

\section{Local error estimate}\label{errorlocal}

As pointed out in the introduction, the formalism of Lie derivatives
does not seem to be easy to use in the framework of this paper, since
the action of the nonlocal operator $\I$ on nonlinearities is rather
involved. As a consequence, we prove an $L^2$ error estimate by a
rather pedestrian way. 

\begin{proposition}[$L^2$ local error estimate]\label{localerror}
Let $u_0 \in H^3(\R)$. There exists $C\left(\|u_0\|_{L^2(\R)} \right)$
such that for all $t \in [0,1]$,
\begin{equation*}
\|Z^t_L u_0 - S^t u_0 \|_{L^{2}(\R)} \le C\left(\|u_0\|_{L^2(\R)}
\right)   t^2\,  \lVert u_0\rVert_{H^3(\R)}^2. 
\end{equation*}
\end{proposition}

\begin{proof}
%\noindent \textbf{First step:} We begin by estimating the term
          %$\|Z^t_L u_0 - S^t u_0 \|_{L^2(\R)}$. \\ 
From the definition of $Z^t_L$ and Remark~\ref{remarkconvolution}, 
we have
\begin{align}
Z^t_L u_0 &= X^t Y^t u_0 = X^t\left( G(t)\ast u_0 - \frac{1}{2}
  \int_0^t  G(t-s) \ast \partial_x(Y^s u_0)^2 \, ds   \right)
\nonumber \\ 
&= D(t) \ast G(t)\ast u_0 -  \frac{1}{2} \int_0^t D(t) \ast  G(t-s)
\ast \partial_x(Y^s u_0)^2 \, ds 
\nonumber \\ 
&= K(t) \ast u_0 -  \frac{1}{2} \int_0^t D(t) \ast
G(t-s) \ast \partial_x(Y^s u_0)^2 \, ds. 
\label{lieformule}
\end{align}
Thus, from Duhamel formula for the Fowler equation \eqref{duhamel5} and the Lie formula \eqref{lieformule}, we have:
\begin{align}
Z_L^t u_0 - S^t u_0 &=   \frac{1}{2} \int_0^t \partial_x  K(t-s) \ast
(S^s u_0)^2  \, ds  
 - \frac{1}{2}  \int_0^t D(t) \ast \partial_x G(t-s) \ast (Y^su_0)^2
 \, ds  \nonumber \\ 
&= \frac{1}{2} \int_0^t \partial_x K(t-s) \ast \left( (S^s u_0)^2 -
  (Z_L^s u_0)^2 \right) \, ds + R(t),  
\label{difference00}
\end{align}
where the remainder $R(t)$ is written as
\begin{equation*}
R(t) =  \frac{1}{2}\int_0^t R_1(s)  ds,\quad \text{with }R_1(s) =   \partial_x
K(t-s) \ast (Z_L^s u_0)^2  -  D(t) \ast \partial_x
G(t-s, \cdot) \ast (Y^su_0)^2. 
\end{equation*} 
Then, from Proposition~\ref{kernel2}, Corollary~\ref{cor:split1} and
Lemma~\ref{estimapriori2}, we have, for $t\in [0,1]$:
\begin{align*}
&\|Z_L^t  u_0 - S^t u_0\|_{L^2(\R)} \le   \frac{1}{2} \int_0^t
\|\partial_x K(t-s, \cdot) \|_{L^2(\R)} \|(S^s u_0)^2 - (Z_L^s u_0)^2
\|_{L^1(\R)} \, ds  \\
&\phantom{\|Z_L^tu_0 - S^t u_0\|_{L^2(\R)} \le} + \|R(t)\|_{L^2(\R)}  \\ 
&\le  C \int_0^t (t-s)^{-3/4} \|S^s u_0-  Z_L^s
u_0\|_{L^2(\R)} \|S^s u_0 +  Z_L^s u_0\|_{L^2(\R)} \, ds +
\|R(t)\|_{L^2(\R)}  \\
&\le  C(e^{\alpha_0 t}+e^{\beta_0 t})  \|u_0\|_{L^{2}(\R)}
\int_0^t (t-s)^{-3/4} \|S^s u_0-  Z_L^s u_0\|_{L^2(\R)} \, ds + \|R(t)\|_{L^2(\R)},
%\label{prapplygronw}
\end{align*}
where $C$ is a positive constant.
To estimate the remainder, we decompose it as follows
\begin{eqnarray*}
R_1(s) & = & T_1 + T_2 + T_3 + T_4,
\end{eqnarray*}
where 
\begin{align*}
T_1 &= K(t-s, \cdot)\ast \partial_x (Z_L^s u_0)^2 - \partial_x (Z_L^s u_0)^2, \\
T_2 &= \partial_x (Y^s u_0)^2 - G(t-s, \cdot) \ast \partial_x (Y^s u_0)^2, \\
T_3 &= G(t-s, \cdot) \ast \partial_x (Y^s u_0)^2-  D(t, \cdot) \ast
G(t-s, \cdot) \ast \partial_x (Y^s u_0)^2, \\ 
T_4 &= \partial_x\left(Z_L^s u_0 \right)^2 - \partial_x\left( Y^s u_0\right)^2. 
\end{align*} 
Let us first study the term $T_1$. From 
Corollaries~\ref{corokernel} and \ref{cor:split1}, we have (recall
that $t\in [0,1]$)
\begin{align*}
\|T_1\|_{L^2(\R)} &= \|K(t-s, \cdot)\ast \partial_x (Z_L^s u_0)^2
- \partial_x (Z_L^s u_0)^2 \|_{L^2(\R)} \\ 
 &\le  C e^{\alpha_0(t-s)} (t-s) \| \partial_x (Z_L^s u_0)^2\|_{H^2(\R)} \\
&\le  C  e^{\alpha_0(t-s)} (t-s) \|Z_L^s u_0\|_{H^3(\R)}^2 \\
&\le  C\left( \|u_0\|_{L^2(\R)} \right)   e^{\alpha_0(t-s)} (t-s)
\|u_0\|_{H^3(\R)}^2.
\end{align*}
In the same way, from Lemma~\ref{lemmechaleur} and
Corollary~\ref{cor:condest}, we control the term $T_2$ as  
\begin{align*}
\|T_2\|_{L^2(\R)} &= \|\partial_x (Y^su_0)^2 - G(t-s, \cdot)
\ast \partial_x (Y^s u_0)^2\|_{L^2(\R)} \\ 
&\le  \varepsilon \, (t-s) \| \partial_x (Y^s u_0)^2\|_{H^2(\R)}
\le  C\left( \|u_0\|_{L^2(\R)} \right) (t-s) \|u_0\|_{H^3(\R)}^2.
\end{align*}
From Lemma \ref{lemme3} and Corollary~\ref{cor:condest}, 
\begin{align*}
\|T_3\|_{L^2(\R)} &= \|G(t-s, \cdot) \ast \partial_x (Y^su_0)^2-  D(t,
\cdot) \ast G(t-s, \cdot) \ast \partial_x (Y^s u_0)^2\|_{L^2(\R)} \\ 
&\le C \, e^{\beta_0 t} \, t \, \|G(t-s, \cdot) \ast \partial_x (Y^s
u_0)^2\|_{H^2(\R)}\\ 
&\le  C \, e^{\beta_0 t} \, t \, \|\partial_x (Y^s u_0)^2\|_{H^2(\R)} \\
&\le  C \, \left( \|u_0\|_{L^2(\R)} \right) \, e^{\beta_0 t} \, t \,
\|u_0\|_{H^3(\R)}^2. 
\end{align*}
For the term $T_4$, write
\begin{align*}
\|T_4\|_{L^2(\R)} &= \|\partial_x\left(Z_L^s u_0 \right)^2
- \partial_x\left( Y^s u_0\right)^2\|_{L^2(\R)} \\ 
&=  2 \|(Z_L^s u_0)  \partial_x\left(Z_L^s u_0 \right) -  (Y^s
u_0) \partial_x\left(Y^s u_0\right)\|_{L^2(\R)}. 
\end{align*}
By linearity of the evolution operator $X^t$, we have
\begin{equation*}
\partial_x\left(Z_L^s u_0 \right) = X^s \partial_x\left(Y^s u_0\right),
\end{equation*}
hence
\begin{align*}
\|T_4\|_{L^2(\R)} &= 2 \|(Z_L^s u_0)  X^s \partial_x\left(Y^s
  u_0\right) -  (Y^s u_0) \partial_x\left(Y^s u_0\right)\|_{L^2(\R)}
\\ 
&\le  2  \left\|X^s \partial_x\left(Y^s u_0\right)\left(X^s Y^s u_0 - Y^s
  u_0 \right) \right\|_{L^2(\R)}  \\
& \quad+ 2  \left\| \left(Y^s u_0\right)\left( X^s \partial_x\left(Y^s
    u_0\right) - \partial_x\left(Y^s u_0\right) \right)
\right\|_{L^2(\R)}.  
\end{align*}
Now from Sobolev embedding, Lemma~\ref{lemme3} and
Corollary~\ref{cor:condest}, we get: 
\begin{align*}
\|T_4\|_{L^2(\R)}
&\le  2  \|X^s \partial_x\left(Y^s u_0\right) \|_{L^{\infty}(\R)} \|X^s Y^s u_0 - Y^s u_0 \|_{L^2(\R)} \\
&\quad+ 2 \| Y^s u_0\|_{L^{\infty}(\R)} \| X^s \partial_x\left(Y^s
  u_0\right) - \partial_x\left(Y^s u_0\right)\|_{L^2(\R)}\\ 
&\le  C \|X^s \partial_x\left(Y^s u_0\right) \|_{H^{1}(\R)} \,
e^{\beta_0 s} \, s \, \|Y^s u_0 \|_{H^2(\R)} \\  
&\quad+ C \| Y^s u_0\|_{H^{1}(\R)}  \, e^{\beta_0 s} \, s \, 
\| \partial_x\left(Y^s u_0\right)\|_{H^2(\R)} \\
&\le  C \, e^{2\beta_0 s} \, s \, \| Y^s u_0\|_{H^{2}(\R)}^2  +
C \, e^{\beta_0 s} \, s \,  \| Y^s u_0\|_{H^{3}(\R)}^2 \\   
&\le  C(\|u_0\|_{L^2(\R)}) \, e^{2\beta_0 s} \, s \, \|u_0\|_{H^3(\R)}^2.
\end{align*}
Finally, since $R_1(s) = T_1+T_2+T_3+T_4$ then for $0\le s\le t\le 1$,
\begin{equation*}
\|R_1(s)\|_{L^2(\R)} \le   C\left( \|u_0\|_{L^2(\R)} \right) \, t \,
\|u_0\|_{H^3(\R)}^2, 
\end{equation*}
and by integration for $s\in [0,t]$,
\begin{equation*}
\|R(t)\|_{L^2(\R)} \le C\left( \|u_0\|_{L^2(\R)} \right) \, \,
t^2 \, \|u_0\|_{H^3(\R)}^2. 
\end{equation*}
We conclude by applying the modified fractional
Gronwall Lemma~\ref{lemme2}. 
\end{proof}

\section{Proof of Theorem~\ref{theoreme1}}
\label{sec:proof}

The proof follows the same lines as
  in \cite[Section~5]{HLR-p}. Denote by $u_k = \(Z_L^{\Delta
    t}\)^ku_0$ the numerical 
  solution,  and
  \begin{equation*}
    u_n^k = S^{(n-k)\Delta t}u_k,
  \end{equation*}
which corresponds to the exact evolution of the numerical value $u_k$
at time $t_k= k\Delta t$ up to time $t_n=n\Delta t$. From
Lemma~\ref{estimapriori2}, there exists $\rho$ such that 
\begin{equation*}
  \|u(t)\|_{H^2(\R)}\le \rho,\quad \forall t\in [0,T].
\end{equation*}
We prove by induction that there exists $\gamma,\Delta t_0,c>0$ such that if
$0<\Delta t\le \Delta t_0$, for
all $n\in \N$ with $n\Delta t\le T$,
\begin{align*}
  &\|u_n\|_{L^2(\R)}\le 2 \rho,\quad
 \|u_n-u(t_n)\|_{L^2(\R)}\le \gamma \Delta t,\\
&\|u_n\|_{H^3(\R)}\le e^{c n\Delta t}\|u_0\|_{H^3(\R)}\le C_0,
\end{align*}
where $C_0= e^{c T}\|u_0\|_{H^3}$. 
The above properties are satisfied for $n=0$. Let $n\ge 1$, and suppose
that the induction assumption is true for $0\le k\le n-1$. 
Since $u_n=u_n^n$ and $u(t_n)= u_n^0$, we estimate
\begin{align*}
  \|u_{n}-u(t_{n})\|_{L^2}&\le \sum_{k=0}^{n-1}
  \|u_{n}^{k+1}-u_{n}^k\|_{L^2} \\
&\le \sum_{k=0}^{n-1}\left\| S^{(n-k-1)\Delta t}\(Z_L^{\Delta t}u_k\)
  - S^{(n-k-1)\Delta t}\(S^{\Delta t}u_k\)\right\|_{L^2}.
\end{align*}
For $k\le n-2$, $Z_L^{\Delta t}u_k= u_{k+1}$ and
Proposition~\ref{prop:lipschitz} yields, along with the induction
assumption, 
\begin{align*}
  \|S^{\Delta t}u_k\|_{L^2} &\le \|S^{\Delta t}u_k - S^{\Delta
    t}u(t_k)\|_{L^2} +\|S^{\Delta t}u(t_k) \|_{L^2}\\
&\le K\|u_k-u(t_k)\|_{L^2} + \|u(t_{k+1})\|_{L^2} \le
K\gamma \Delta t +\rho,
\end{align*}
which is bounded by $2\rho$ if $0<\Delta t\le \Delta t_0\ll 1$. Up to
replacing $K$ with $\max(K,1)$, we obtain, for $k\le n-1$ and $n\Delta
t\le T$, 
\begin{align*}
  \left\| S^{(n-k-1)\Delta t}\(Z_L^{\Delta t}u_k\)
  - S^{(n-k-1)\Delta t}\(S^{\Delta t}u_k\)\right\|_{L^2}\le K
\|Z_L^{\Delta t}u_k - S^{\Delta t}u_k\|_{L^2}.
\end{align*}
Using Proposition~\ref{localerror}, we infer
\begin{align*}
  \left\| S^{(n-k-1)\Delta t}\(Z_L^{\Delta t}u_k\)
  - S^{(n-k-1)\Delta t}\(S^{\Delta t}u_k\)\right\|_{L^2}\le C K(\Delta
t)^2 e^{2ck\Delta t}\|u_0\|_{H^3}^2,
\end{align*}
for some uniform constant $C$. 
Therefore,
\begin{align*}
  \|u_{n}-u(t_{n})\|_{L^2}&\le nC K(\Delta
t)^2 e^{2cT}\|u_0\|_{H^3}^2\le CTKe^{2cT} \Delta t,
\end{align*}
which yields the first two estimates of the induction,
provided one takes $\gamma = CTKe^{2cT}$, 
which is uniform in $n$ and $\Delta t$. Finally, the last estimate of
the induction stems from Corollary~\ref{cor:split1}.

\section{Numerical Experiments \label{numerique}} 

The aim of this section is to numerically verify the Lie method
convergence rate in $\mathcal{O}\left( \Delta t \right) $ for the
Fowler equation \eqref{fowlereqn5}. \\ 
To solve the linear sub-equation \eqref{nlocal}, discrete Fourier
transform is used  and for the nonlinear sub-equation \eqref{burgers},
different numerical approximations can be used. Here, we use the
finite difference method.\\ 
Since the discrete Fourier transform plays a key role in these
schemes, we briefly review its definition, which can be found in most
books. In some situation, when the mesh nodes number $N$ is chosen to
be $N=2^p$ for some integer $p$, a fast Fourier transform (FFT)
algorithm is used to further decrease the computation time. In this
work we will use a subroutine implemented in Matlab.  In this program,  
the interval $[0,1]$ is discretized by $N$ equidistant points, with
spacing $\Delta x = 1/N$. The spatial grid points are then given by
$x_j = j/N$, $j=0,...,N$. If $u_j(t)$ denotes the approximate
solution to $u(t,x_j)$, the discrete Fourier transform of the sequence
$ \left\lbrace u_j \right\rbrace_{j = 0}^{N-1} $ is defined by 
\begin{equation*}
\hat{u}(k) = \mathcal{F}^d_k (u_j) = \sum_{j=0}^{N-1} u_j e^{-2i \pi j k /N },
\end{equation*}
for $k = 0,\cdots, N-1$, and the inverse discrete Fourier transform is given by
\begin{equation*}
u_j = \mathcal{F}_{j}^{-d}( \hat{u}_k ) = \frac{1}{N} \sum_{k = 0}^{N -1} \hat{u}_k e^{ 2 i \pi k x_j}, 
\end{equation*}
for $j = 0, \cdots,N-1 $. Here $\mathcal{F}^d$ denotes the discrete
Fourier transform and $\mathcal{F}^{-d}$ its inverse.
\bigbreak

In what follows, the linear equation \eqref{nlocal} is solved using the discrete Fourier transform and time marching is performed exactly according to 
\begin{equation}
u_j^{n+1} = \F^{-d}_j \left(e^{ -\phi_\I(k)\Delta t } \F^{d}_k (u_j^n) \right).
\end{equation}
To approximate the viscous Burgers' equation \eqref{burgers}, we use the
following explicit centered scheme:
\begin{equation}\label{schemeburgers}
u_j^{n+1} = u_j^{n} - \frac{\Delta t}{2\Delta x} \left[
  \left(\frac{u^2}{2} \right)_{j+1}^{n} -  \left(\frac{u^2}{2}
  \right)_{j-1}^{n} \right] + \varepsilon \, \Delta t\frac{u_{j+1}^n-2
  u^n_j + u_{j-1}^n }{\Delta x^2}, 
\end{equation}
which is stable under the CFL-Peclet condition
%According to the linear analysis, the splitting methods are linearly stable under the CFL-Peclet condition
\begin{equation}
\Delta t = \min \left(\frac{\Delta x}{|v|}, \frac{\Delta x^2}{2 \varepsilon} \right), 
\label{CFL}
\end{equation} 
where $v$ is an average value of $u$ in the neighbourhood of $(t^n , x_j $).

\begin{remark}
In the case where the linear sub-equation \eqref{nlocal} is solved using a finite difference scheme 
instead of a FFT computation, an additional stability condition is
required, see \cite{AB-p}. Moreover, the computation time becomes very
long because of the discretization of the nonlocal term which is
approximated using a quadrature rule. Indeed, in \cite{AB-p}, the
Fowler equation has been discretized using finite difference method
and the numerical analysis showed that this operation is
computationally expensive. This observation has also motivated the use
of splitting methods, in particular the implementation of split-step
Fourier methods.  
\end{remark}
%We have proved that the Lie formulation is of order one in time for initial data in $H^3$. 

In order to avoid numerical reflections due to boundaries conditions
and to justify the use of the FFT method, we consider initial data
with compact support displayed in Figure~\ref{cinitiale} to perform
numerical simulations. 

%To perform numerical simulations, we use initial data displayed in Figure \ref{cinitiale}. 
\begin{figure}[h!]
	\centering
	\includegraphics[scale=0.3]{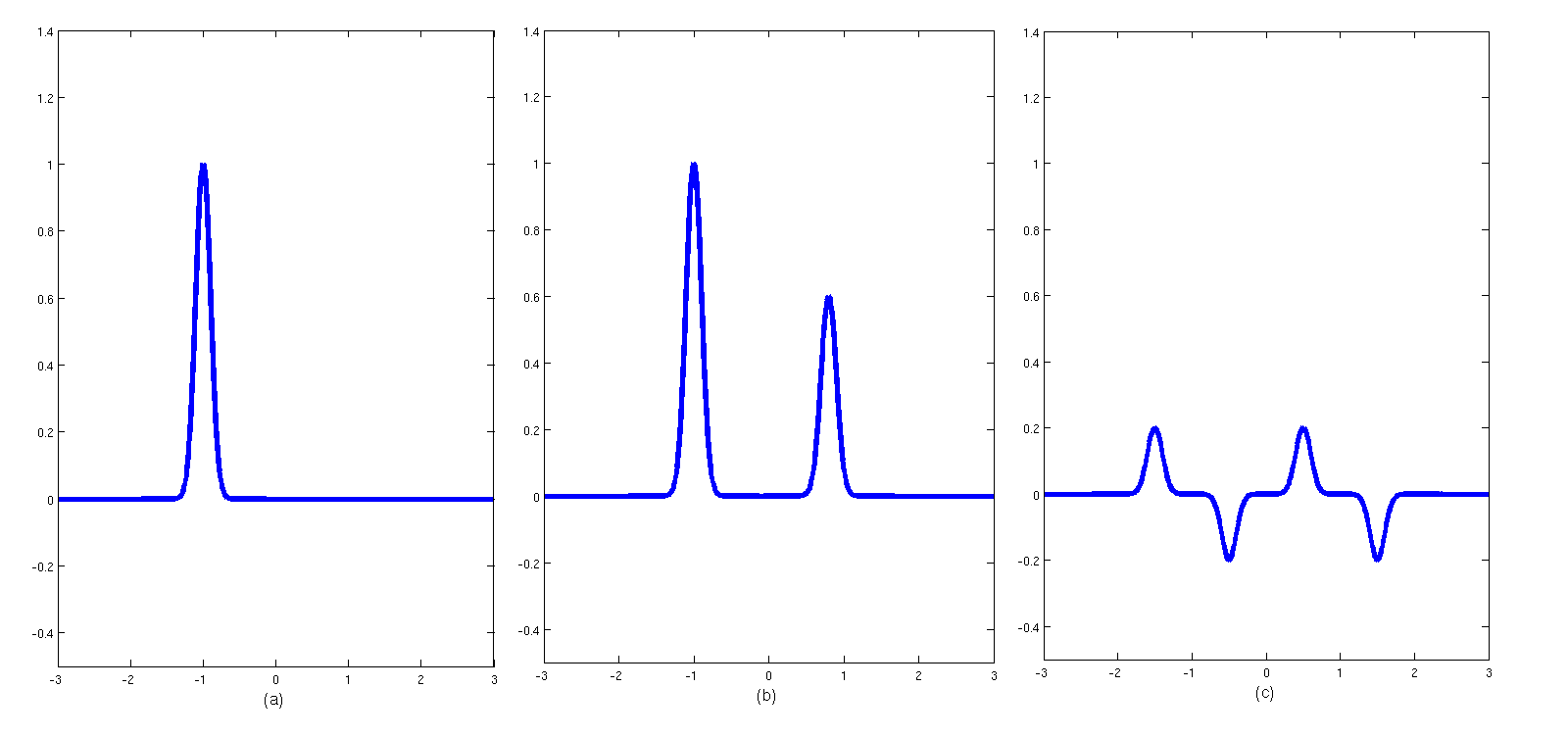} 
    \vspace{-1 cm} 	
	\caption{Initial data used for numerical experiments.}
	%The dotted line is reference with first-order slope.}
	\label{cinitiale}	
\end{figure} 

Since we do not know the exact solution of the Fowler
  equation, a classical numerical way to determine the convergence
  numerical order of schemes is to plot the logarithm of the error
  $||u_1(T) -u_2(T)||_{L^2}$ in function of the logarithm of the step
  time $\Delta t$,   
%To determine the numerical order, we consider the following number
%\begin{equation}
%E_{2} = \frac{1}{N^2}  \sum_{n=0}^{N}   | u^1_j(T) - u^2_j(T) |^2,
%\label{errol2}
%\end{equation}}
where $u_1$ and  $u_2$ are computed for time steps
$\Delta t/2$ and $\Delta t/4$, respectively, up to the final time
$T$. Hence, the 
numerical order corresponds to the slope of the curve, see Figures
\ref{ordcvglie}, \ref{ordcvgstrang}. For reference, a small line of
slope one (resp. two) is added in Figure \ref{ordcvglie}
(resp. \ref{ordcvgstrang}). We see that the slopes for the three
initial data match well and so we can conclude that numerical
simulations are consistent with the theoretical results established
above. 
%Nevertheless, to get these results, we need regularity for
%initial data. For instance, for the Lie method, from Theorem
%\ref{theoreme1} initial data must be at least in $H^3$.   
 \\
%Hence, the numerical order corresponds to the slope of the $\log(E_1)$
%plotted in function of $\Delta t$, see Figure \ref{ordcvg}. \\
%Another way to get numerical order is to compute
%\begin{equation*} 
%p = \max_{t_n \in [0,T]} \frac{1}{\ln 2} \ln \left( \frac{||u_1 - u_2||_{L^2}}{||u_2 - u_3||_{L^2}} \right), 
%\end{equation*}
%where $t_n = n \Delta t$ and $u_1, u_2, u_3$ are respectively computed for time steps $\Delta t$, $\Delta/2$ and $\Delta t/4$. 

\begin{figure}[ht!]
	\centering
	\includegraphics[scale=0.3]{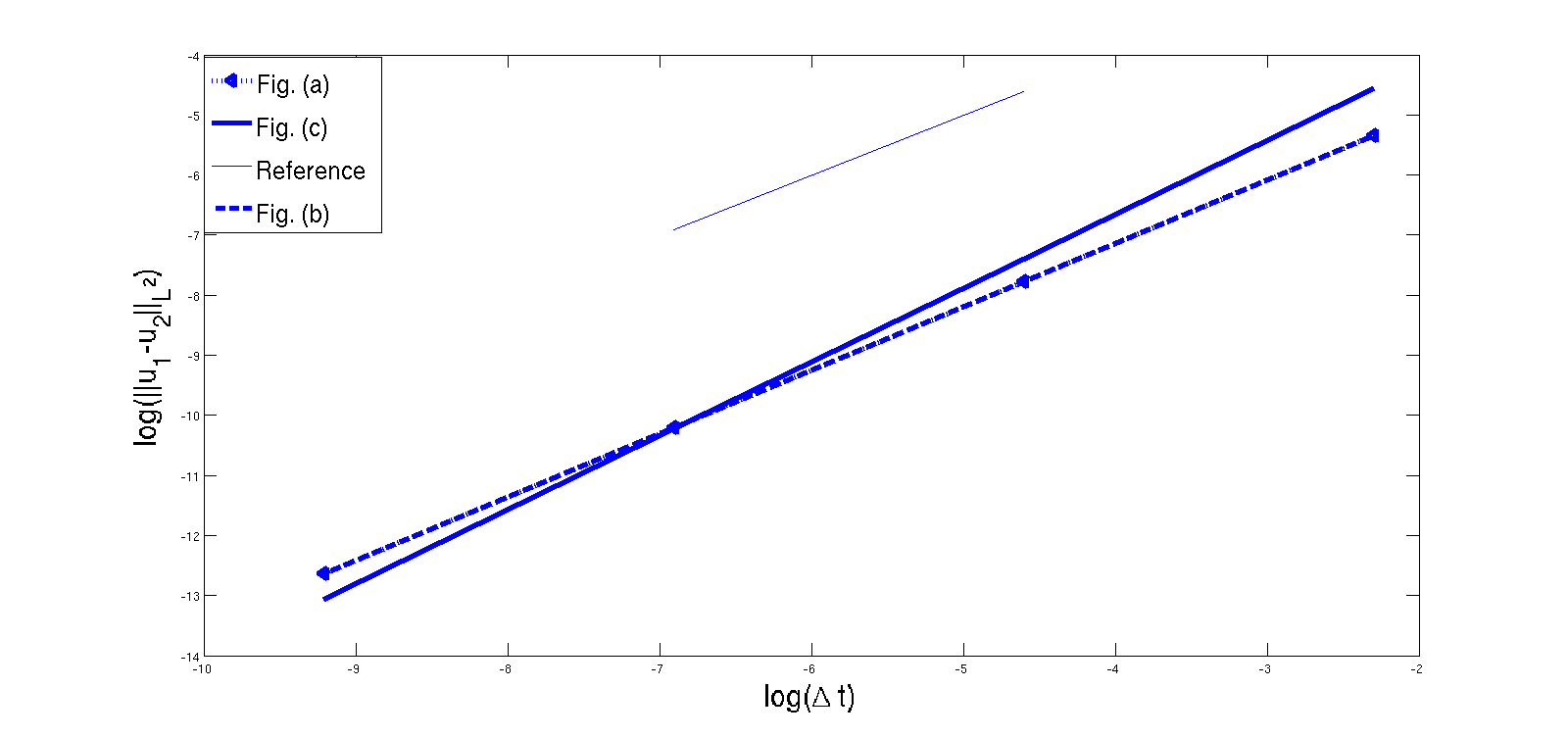} 
   \vspace{-0.8 cm} 	
	\caption{Lie method}
	%The dotted line is reference with first-order slope.}
	\label{ordcvglie}	
\end{figure} 

We also study numerical convergence of Strang splittings using initial data displayed in Figure \ref{cinitiale}. Results are plotted in Figure \ref{ordcvgstrang}. 
We can see that the Strang formulation is of order two in time for smooth initial data. 

\begin{figure}[ht!]
	\centering
	\includegraphics[scale=0.3]{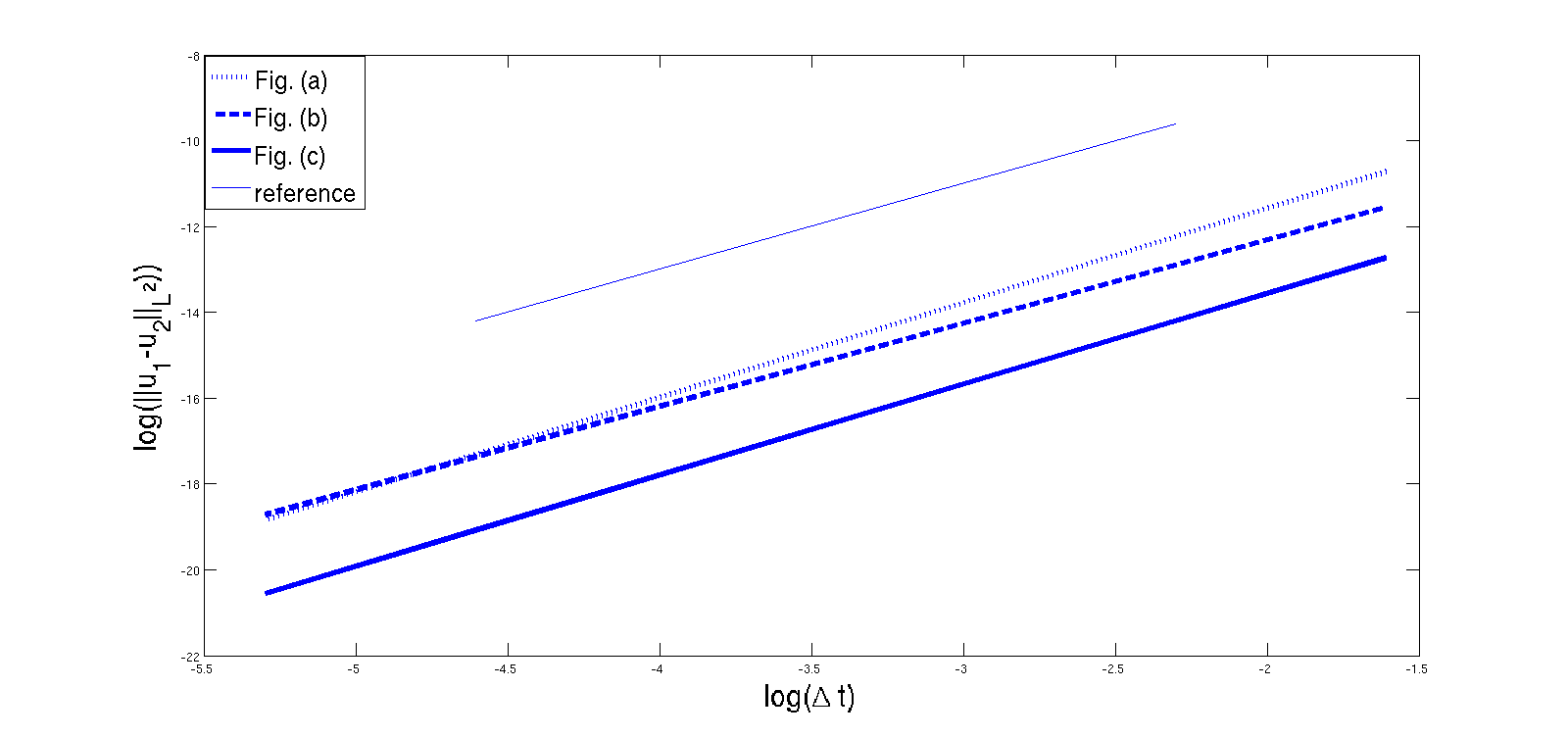} 
    \vspace{-0.8 cm} 	
	\caption{Strang method}
	%The dotted line is reference with first-order slope.}
	\label{ordcvgstrang}	
\end{figure} 
From numerical simulations, we emphasize the fact that both formulas defining a Lie operator, as well as both
formulas defining a Strang operator, lead to the same results.

%Table
%\ref{table1} displays the results for different CFL. We emphasize the
%fact that both formulas defining a Lie operator, as well as both
%formulas defining a Strang operator, lead to the same results. 
%\begin{table}[htbp]
%\begin{center}
%\begin{tabular}{|*{3}{c|}}
%  \hline
%    & Lie & Strang \\
%  \hline
%     CFL = 0.5  &    1.001   &   2.0081     \\
%  \hline
%     CFL  = 1.0  &  1.0004 & 2.0002\\
%   \hline
%\end{tabular}
%\end{center}
%\caption{Numerical order of accuracy. }
%\label{table1}
%\end{table}

%\begin{figure}[ht!]
%	\centering	
%        \subfigure[Lie method. Dotted line has slope one. ]
%	{\includegraphics[scale = 0.3]{convDTLie2.png}}
%	\subfigure[ Strang method. Dotted line has slope two.]
%	{\includegraphics[scale=0.3]{constrangcfl1.png}}
%        \caption{Numerical convergence rates in $\Delta t$ and $\Delta t^2$. 
%}
%\label{ordcvg}
%\end{figure} 

\subsection*{Acknowledgements} The authors are grateful to
Pascal Azerad and Bijan Mohammadi for helpful comments.


\begin{thebibliography}{10}

\bibitem{AAI10}
{\sc N.~Alibaud, P.~Azerad, and D.~Is{\`e}be}, {\em A non-monotone nonlocal
  conservation law for dune morphodynamics}, Differential Integral Equations,
  23 (2010), pp.~155--188.

\bibitem{AA09}
{\sc B.~Alvarez-Samaniego and P.~Azerad}, {\em Existence of travelling-wave
  solutions and local well-posedness of the {F}owler equation}, Discrete
  Contin. Dyn. Syst. Ser. B, 12 (2009), pp.~671--692.

\bibitem{AB-p}
{\sc P.~Azerad and A.~Bouharguane}, {\em Finite difference approximations for a
  fractional diffusion/anti-diffusion equation}.
\newblock preprint, \url{http://arxiv.org/abs/1104.4861}, 2011.

\bibitem{ABC12}
{\sc P.~Azerad, A.~Bouharguane, and J.-F. Crouzet}, {\em Simultaneous denoising
  and enhancement of signals by a fractal conservation law}, Commun. Nonlinear
  Sci. Numer. Simul., 17 (2012), pp.~867--881.

\bibitem{BBD02}
{\sc C.~Besse, B.~Bid{\'e}garay, and S.~Descombes}, {\em Order estimates in
  time of splitting methods for the nonlinear {S}chr\"odinger equation}, SIAM
  J. Numer. Anal., 40 (2002), pp.~26--40.

\bibitem{Bo12}
{\sc A.~Bouharguane}, {\em On the instability of a nonlocal conservation law},
  Discrete Contin. Dyn. Syst. Ser. S, 5 (2012), pp.~419--426.

\bibitem{DeTh-p}
{\sc S.~Descombes and M.~Thalhammer}, {\em The {L}ie--{T}rotter splitting for
  nonlinear evolutionary problems with critical parameters. {A} compact local
  error representation and application to nonlinear {S}chr{\"o}dinger equations
  in the semi-classical regime}, IMA J. Numer. Anal.,  (2012).
\newblock to appear.

\bibitem{Diethelm2004}
{\sc K.~Diethelm and N.~J. Ford}, {\em Multi-order fractional differential
  equations and their numerical solution}, Applied Mathematics and Computation,
  154 (2004), pp.~621 -- 640.

\bibitem{DI04}
{\sc J.~Droniou and C.~Imbert}, {\em Solutions de viscosit\'e et solutions
  variationnelles pour {EDP} non-lin\'eaires}.
\newblock Lecture notes, available at
  \url{http://www-gm3.univ-mrs.fr/cours/fichiers/cours-dea.pdf}, 2004.

\bibitem{Fa09}
{\sc E.~Faou}, {\em Analysis of splitting methods for reaction-diffusion
  problems using stochastic calculus}, Math. Comp., 78 (2009), pp.~1467--1483.

\bibitem{Fo-p}
{\sc A.~C. Fowler}, {\em Mathematics and the environment}.
\newblock Lecture notes, available at
  \url{http://www2.maths.ox.ac.uk/~fowler/courses/mathenvo.html}.

\bibitem{Fo01}
\leavevmode\vrule height 2pt depth -1.6pt width 23pt, {\em Dunes and drumlins},
  in Geomorphological fluid mechanics, A.~Provenzale and N.~Balmforth, eds.,
  vol.~211, Springer-Verlag, Berlin, 2001, pp.~430--454.

\bibitem{HKLR10}
{\sc H.~Holden, K.~H. Karlsen, K.-A. Lie, and N.~H. Risebro}, {\em Splitting
  methods for partial differential equations with rough solutions}, EMS Series
  of Lectures in Mathematics, European Mathematical Society (EMS), Z\"urich,
  2010.
\newblock Analysis and MATLAB programs.

\bibitem{HKRT11}
{\sc H.~Holden, K.~H. Karlsen, N.~H. Risebro, and T.~Tao}, {\em Operator
  splitting for the {K}d{V} equation}, Math. Comp., 80 (2011), pp.~821--846.

\bibitem{HLR-p}
{\sc H.~Holden, C.~Lubich, and N.~H. Risebro}, {\em Operator splitting for
  partial differential equations with {B}urgers nonlinearity}, Math. Comp.,
  (2012).
\newblock To appear. Archived as \url{http://arxiv.org/abs/1102.4218}.

\bibitem{KaPo88}
{\sc T.~Kato and G.~Ponce}, {\em Commutator estimates and the {E}uler and
  {N}avier-{S}tokes equations}, Comm. Pure Appl. Math., 41 (1988),
  pp.~891--907.

\bibitem{Lu08}
{\sc C.~Lubich}, {\em On splitting methods for {S}chr\"odinger-{P}oisson and
  cubic nonlinear {S}chr\"odinger equations}, Math. Comp., 77 (2008),
  pp.~2141--2153.

\bibitem{RS09}
{\sc D.~L. Ropp and J.~N. Shadid}, {\em Stability of operator splitting methods
  for systems with indefinite operators: advection-diffusion-reaction systems},
  J. Comput. Phys., 228 (2009), pp.~3508--3516.

\bibitem{Sa07}
{\sc A.~Sacchetti}, {\em Spectral splitting method for nonlinear
  {S}chr\"odinger equations with singular potential}, J. Comput. Phys., 227
  (2007), pp.~1483--1499.

\bibitem{St68}
{\sc G.~Strang}, {\em On the construction and comparison of difference
  schemes}, SIAM J. Numer. Anal., 5 (1968), pp.~506--517.

\bibitem{TA84}
{\sc T.~R. Taha and M.~J. Ablowitz}, {\em Analytical and numerical aspects of
  certain nonlinear evolution equations. {II}. {N}umerical, nonlinear
  {S}chr\"odinger equation}, J. Comput. Phys., 55 (1984), pp.~203--230.

\bibitem{Ye2007}
{\sc H.~Ye, J.~Gao, and Y.~Ding}, {\em A generalized gronwall inequality and
  its application to a fractional differential equation}, Journal of
  Mathematical Analysis and Applications, 328 (2007), pp.~1075 -- 1081.

\end{thebibliography}
\end{document}